\documentclass[11pt]{article}
\usepackage{amssymb}
\usepackage{amsmath}
\usepackage{graphicx}
\usepackage{setspace}

\begin{document}

\newtheorem{theorem}{Theorem}[section]
\newtheorem{conjecture}[theorem]{Conjecture}
\newtheorem{corollary}[theorem]{Corollary}
\newtheorem{lemma}[theorem]{Lemma}
\newtheorem{claim}[theorem]{Claim}
\newtheorem{proposition}[theorem]{Proposition}
\newtheorem{construction}[theorem]{Construction}
\newtheorem{definition}[theorem]{Definition}
\newtheorem{question}[theorem]{Question}
\newtheorem{problem}[theorem]{Problem}
\newtheorem{remark}[theorem]{Remark}
\newtheorem{observation}[theorem]{Observation}

\newcommand*{\qed}{\hfill\ensuremath{\blacksquare}}%
\newcommand{\ex}{{\mathrm{ex}}}

\newcommand{\EX}{{\mathrm{EX}}}

\newcommand{\AR}{{\mathrm{AR}}}

\def\endproofbox{\hskip 1.3em\hfill\rule{6pt}{6pt}}
\newenvironment{proof}%
{%
\noindent{\it Proof.}
}%
{%
 \quad\hfill\endproofbox\vspace*{2ex}
}
\def\qed{\hskip 1.3em\hfill\rule{6pt}{6pt}}
\def\ce#1{\lceil #1 \rceil}
\def\fl#1{\lfloor #1 \rfloor}
\def\lr{\longrightarrow}
\def\e{\varepsilon}
\def\ex{{\rm\bf ex}}
\def\cA{{\cal A}}
\def\cB{{\cal B}}
\def\cC{{\cal C}}
    \def\cD{{\cal D}}
\def\cF{{\cal F}}
\def\cG{{\cal G}}
\def\cH{{\cal H}}
\def\cK{{\cal K}}
\def\cI{{\cal I}}
\def\cJ{{\cal J}}
\def\cL{{\cal L}}
\def\cM{{\cal M}}
\def\cP{{\cal P}}
\def\cQ{{\cal Q}}
\def\cR{{\cal R}}
\def\cS{{\cal S}}
\def\cE{{\cal E}}
\def\cT{{\cal T}}
\def\ex{{\rm ex}}
\def\pr{{\rm Pr}}
\def\exp{{\rm  exp}}
\def\va{\vec{a}}

\def\wt{\widetilde{T}}
\def\bkl{{\cal B}^{(k)}_\ell}
\def\cmkt{{\cal M}^{(k)}_{t+1}}
\def\cpkl{{\cal P}^{(k)}_\ell}
\def\cckl{{\cal C}^{(k)}_\ell}
\def\pkl{\mathbb{P}^{(k)}_\ell}
\def\ckl{\mathbb{C}^{(k)}_\ell}

\def\mC{{\cal C}}

\def\imp{\Longrightarrow}
\def\1e{\frac{1}{\e}\log \frac{1}{\e}}
\def\ne{n^{\e}}
\def\rad{ {\rm \, rad}}
\def\equ{\Longleftrightarrow}
\def\pkl{\mathbb{P}^{(k)}_\ell}

\def\mE{\mathbb{E}}

\def\mP{\mathbb{P}}

\def \e{\epsilon}
\voffset=-0.5in

\setstretch{1.1}
\pagestyle{myheadings}
\markright{{\small \sc  Jiang, Qiu:}
	{\it\small On Tur\'an numbers of bipartite subdivisions}}

\title{\huge\bf  Tur\'an numbers of bipartite subdivisions}

\author{
	Tao Jiang\thanks{Department of Mathematics, Miami University, Oxford,
		OH 45056, USA. E-mail: jiangt@miamioh.edu. }
	\quad \quad Yu Qiu \thanks{
		School of Mathematical Sciences,
		University of Science and Technology of China, Hefei, 230026,
		P.R. China. Email: yuqiu@mail.ustc.edu.cn. Research supported by China Scholarship Council.
		\newline\indent
		{\it 2010 Mathematics Subject Classifications:}
		05C35.\newline\indent
		{\it Key Words}:  Tur\'an number,  Tur\'an exponent, extremal function, subdivision
} }

\date{May 28, 2019}
\maketitle
\begin{abstract}
	Given a graph $H$, the {\it Tur\'an number} $\ex(n,H)$ is the largest number of edges in an $H$-free graph on $n$ vertices. We make progress on a recent conjecture of Conlon, Janzer, and Lee \cite{CJL} on the Tur\'an numbers of bipartite graphs, which in turn yields further progress on a conjecture of Erd\H{o}s and Simonovits \cite{Erdos}.
	
	Let $s,t,k\geq 2$ be integers. Let $K_{s,t}^k$ denote the graph obtained from the complete bipartite graph $K_{s,t}$	by replacing each edge $uv$ in it with a path of length $k$ between $u$ and $v$ such that the $st$ replacing paths are internally disjoint.
	It follows from a general theorem of Bukh and Conlon \cite{BC} that
	$\ex(n,K_{s,t}^k)=\Omega(n^{1+\frac{1}{k}-\frac{1}{sk}})$.
	Conlon, Janzer, and Lee \cite{CJL} recently conjectured that
	for any integers $s,t,k\geq 2$, $\ex(n,K_{s,t}^k)=O(n^{1+\frac{1}{k}-\frac{1}{sk}})$. Among many other things,  they settled the $k=2$ case of their conjecture.  As the main result of this paper, we prove their conjecture for $k=3,4$.

	Our main results also yield infinitely many new so-called {\it Tur\'an exponents}:
	rationals $r\in (1,2)$ for which there exists a bipartite graph $H$ with $\ex(n, H)=\Theta(n^r)$,
	adding to the lists recently obtained by Jiang, Ma, Yepremyan \cite{JMY}, by Kang, Kim, Liu
	\cite{KKL}, and by Conlon, Janzer, Lee \cite{CJL}. 
	
Our method builds on  an extension of the Conlon-Janzer-Lee method. We also note that the extended method also gives a weaker version of the Conlon-Janzer-Lee conjecture for all $k\geq 2$.

\end{abstract}

\section{Introduction}
Given a family $\cH$ of graphs, the {\it Tur\'an number} $ex(n,\cH)$ is the largest number of edges
in an $n$-vertex graph that does not contain any member of $\cH$. If $\cH$ consists of a single graph $H$, we write $ex(n,H)$ for $ex(n,\{H\})$.  Let $p=\min\{\chi(H)-1: H\in \cH\}$,
where $\chi(H)$ denotes the chromatic number of $H$. The celebrated Erd\H{o}s-Stone-Simonovits theorem asserts that $\ex(n,\cH)=(1-\frac{1}{p}+o(1))\binom{n}{2}$. This determines the function for all families that do not contain a bipartite member. When $\cH$ contains a bipartite graph, the problem
is generally wide-open, with many intriguing conjectures. One of these, known as the {\it Tur\'an exponent conjecture}, was made by Erd\H{o}s and Simonovits \cite{Erdos} that asserts that for any rational $r\in (1,2)$ there exists a bipartite graph $H$ such that $\ex(n,H)=\Theta(n^r)$. We call a rational $r$ for which the Erd\H{o}s-Simonovits conjecture holds a {\it Tur\'an exponent}.
In a recent breakthrough,
Bukh and Conlon \cite{BC} have proved that for any rational number $r\in (1,2)$ there exists a finite family $\cH$ of graphs such that $\ex(n,\cH)=\Theta(n^r)$. On the other hand, the original conjecture of Erd\H{o}s and Simonovits concerning single bipartite graphs is still generally open. Until recently, it was only known to be true for $r=1+1/k$ and $r=2-1/k$ where $k\geq 2$ is a positive integer. Recently, there have been a flurry of progresses on the conjecture, by Jiang, Ma, Yepremyan \cite{JMY}, by Kang, Kim, Liu
\cite{KKL}, and by Conlon, Janzer, Lee \cite{CJL}. For more detailed discussions on recent works
on the Erd\H{o}s-Simonovits conjecture, the reader is referred to \cite{BC, JMY, KKL, CJL}.

A recent focal point on the Erd\H{o}s-Simonovits conjecture, with motivations from other problems as well, concerns the Tur\'an number of so-called subdivisions of graphs. Given a graph $H$, and an integer $k\geq 2$, let $H^k$ denote the graph obtained by replacing each edge $uv$ of $H$ with a path 
of length $k$ between $u$ and $v$ so that the $e(H)$ replacing paths are internally vertex disjoint. The Tur\'an number of $H^k$ is
studied in \cite{Jiang} and \cite{JS}, based on earlier work in \cite{KP}. Recently, significant progresses  on the problem have been made in \cite{CL}, \cite{Janzer}, and \cite{CJL}.
Let $s,t,k\geq 2$ be integers.  As usual, let $K_{s,t}$ denote the complete bipartite graph with part sizes $s$ and $t$. Let $K_{s,t}^{k}=(K_{s,t})^k$.
It follows from the above mentioned breakthrough work of Bukh and Conlon \cite{BC} that
$\ex(n,K_{s,t}^k)=\Omega(n^{1+\frac{1}{k}-\frac{1}{sk}})$.
Conlon, Janzer, and Lee \cite{CJL} recently made the following conjecture
on a matching upper bound.

\begin{conjecture} {\rm \cite{CJL}} \label{CJL-conjecture}
	For any integers $s,t,k\geq 2$, $\ex(n,K_{s,t}^k)=O(n^{1+\frac{1}{k}-\frac{1}{sk}})$.
\end{conjecture}

In \cite{CJL},  among many other things, Conlon, Janzer and Lee settled the $k=2$ case of Conjecture \ref{CJL-conjecture}, showing that
$\ex(n, K_{s,t}^2)=O(n^{\frac{3}{2}-\frac{1}{2s}})$. In this paper, we prove their conjecture
for $k=3,4$. 
\begin{theorem}\label{main}
	For any integers $s,t\geq 2$ and $k\in \{3,4\}$, $\ex(n,K_{s,t}^k)
	=O(n^{1+\frac{1}{k}-\frac{1}{sk}})$.
\end{theorem}

We remark that our theorem together with the theorem of Bukh and Conlon also yields infinitely many new {\it Tur\'an exponents}: namely those of the form $1+\frac{1}{k}-\frac{1}{sk}$, where
$s\geq 2$ is any integer and $k\in \{3,4\}$. The majority of the rest of the paper is devoted to the proof of our main result: Theorem \ref{main}. We then conclude with some observations in the concluding remarks.


\section{Notation and a basic lemma}

As is often the case in the study of bipartite Tur\'an problems,
our problem may be reduced to the setting in which the host graph is almost regular.
Specifically, given a positive integer $K$, we say that a graph $G$ is {\it $K$-almost-regular}
if $\Delta(G)\leq K\cdot \delta(G)$.

The following lemma can be found in \cite{JS}, which is a slight adaption
of the regularization lemma of Erd\H{o}s and Simnovits \cite{cube}. Another recent adaption of this can be found in \cite{CJL}.

\begin{lemma}{\rm (\cite{JS}] Proposition 2.7)} \label{almost-regular}
	Let $0<\epsilon<1$ and $c\ge1$. There exists $n_0=n_0(\epsilon)>0$ such that the following holds for all $n\ge n_0$. If $G$ is a graph on $n$ vertices with $e(G)\ge cn^{1+\epsilon}$, then $G$ contains a $K$-almost regular subgraph $G'$ on $m\ge n^{\frac{\epsilon-\epsilon^2}{2+2\epsilon}}$ vertices such that $e(G')\ge\frac{2c}5m^{1+\epsilon}$ and $K=20\cdot2^{\frac1{\epsilon^2}+1}.$
\end{lemma}

For most of the rest of the paper we will always assume our host graph $G$ to be almost regular.
Then in the main proof we apply Lemma \ref{almost-regular} on general host graphs.

\section{Building subdivisions using  substructures}

\subsection{Building subdivisions using critical paths} \label{critical-paths} 
 
In this section, we present one of the main ingredients used by Conlon, Janzer, and Lee \cite{CJL}.
To make our presentation  consistent with the rest of our paper, we present their
results using our notation and terminology.

\begin{definition} {\rm (Definition 6.2 of \cite{CJL})} 
	{\rm
		Let $L$ be an integer. Define the function $f(\ell, L)$ for $0\leq\ell\leq k$ recursively by
		setting $f(0,L)=1,f(1,L)=L$ and, for $2\leq \ell\leq k$,
		\[f(\ell, L)=1+f(\ell-1)^{16} (\ell-1)^2 \max_{1\leq i\leq \ell-1} f(i,L) f(\ell-i, L).\]
	}
\end{definition}
\begin{definition} \label{heavy-light-path-definition}
	{\rm
		Let $L$ be fixed. Let $G$ be a graph. For each $\ell\geq 1$,
		a path $P$ of length $\ell$ in $G$ with endpoints $x,y$ is called
		{\it $\ell$-heavy} if there are more than $f(\ell,L)$ distinct $x,y$-paths of length $\ell$ in $G$ and
		is called {\it $\ell$-light} otherwise.
		A path $P$ of length $\ell$ in $G$ is called {\it $\ell$-critical} if it is $\ell$-heavy but for each $j<\ell$
		each subpath of length $j$ is $j$-light. Since the length of a given path is fixed, we may drop the prefix and use terms heavy, light, critical directly.}
\end{definition}
In \cite{CJL}, lights paths are called {\it good paths} and critical paths are called {\it admissible paths}.
The following lemma is implied by Lemma 6.8  and Corollary 6.9 of \cite{CJL} since their forbidden subgraph $H$ is a supergraph of $K_{s,t}^k$.
\begin{lemma} \label{short-paths}
	Let $G$ be a $K_{s,t}^{k}$-free 
	$K$-almost-regular graph on $n$ vertices with minimum degree $\delta$.
	Then provided that $L$ is sufficiently large compared to $s,t,k$ and $K$, for any
	$2\leq \ell \leq k$, the number of $\ell$-critical paths is at most $n\frac{2(K\delta)^\ell}{f(\ell-1, L)}$. \qed
\end{lemma}

Lemma \ref{short-paths} roughly says if a graph has many short critical paths, then we can
easily build a copy of $K_{s,t}^k$. In the next subsection, we develop analogous statements
for other critical substructures.


\subsection{Building subdivisions using strong spiders, the general case} \label{general-case}

A {\it non-path spider} is a tree with exactly one vertex $w$ of degree at least three, called the {\it center}. Paths from the center to the leaves are called {\it legs}. 
A spider in which all legs have length $h$ is called a {\it balanced spider of height $h$}.
In this section, whenever we discuss a non-path spider $T$ with $m$ legs, we always fix a particular labelling of its leaves as $v_1,\dots, v_m$. For each $i\in [m]$, let $\ell_i$ be distance from the center $w$ of $T$ to $v_i$. We call  $T$ a spider with {\it leaf vector} $(v_1,\dots, v_m)$ and 
{\it length vector} $(\ell_1,\dots, \ell_m)$.

\begin{definition} \label{strong-definition}
	{\rm  
	Let $s,k\geq 2$ be integers.
		Let $G$ be a graph. Let $\vec{\ell}:=(\ell_1,\dots, \ell_s)$ be a vector of $s$ positive integers,
		each of which is at most $k$. 
		We say that a vertex ordered tuple $(v_1,\dots, v_s)$ in $G$ is 
		$(\ell_1,\dots, \ell_s)$-{\it strong}  if $G$ contains at least $(sk)^{sk-\ell}\cdot f(k,L)$ internally vertex-disjoint  spiders with leaf vector $(v_1,\dots, v_s)$ and length vector $(\ell_1,\dots, \ell_s)$,
		where $\ell=\ell_1+\dots +\ell_s$. A spider with leaf vector $(v_1,\dots, v_s)$ and length vector $(\ell_1,\dots, \ell_s)$ is called $(\ell_1,\dots, \ell_s)$-{\it strong} if the tuple $(v_1,\dots, v_s)$ is $(\ell_1,\dots, \ell_s)$-strong. Since the length vector of any spider is fixed, whenever we say a spider is strong, it is understood that it is strong relative to its length vector.
	}
\end{definition}

\begin{lemma} \label{spider-shrinking}
	Let $G$ be a $K$-almost-regular graph with minimum degree $\delta$.
	Let $x$ be a vertex. Let $\cC$ be a family
	of at least $\alpha \delta^h$ distinct paths of length $h$ with one end $x$  and
	another end in a set $S$. Then $\cC$ contains a subfamily $\cD$ of more than
	$(\alpha/h K^{h-1}) \delta$ paths which are vertex-disjoint outside $\{x\}$.
\end{lemma}
\begin{proof}
	Let $\cD\subseteq \cC$ be a maximal subfamily of paths that are vertex disjoint outside $\{x\}$.
	Let $W$ be the set of vertices contained in these paths besides $x$. Then $|W|=h|\cD|$.
	By the maximality of $\cD$ each member of $\cC$ must pass through $x$ and  some vertex in $W$.
	Since $G$ has maximum degree at most $K\delta$, there can be at most $|W|(K\delta)^{h-1}$ such paths. Hence $|\cC|\leq |W|(K\delta)^{h-1}$. Since $|\cC|\geq \alpha \delta^h$ and $|W|=h|\cD|$, we
	have $|\cD|>(\alpha/hK^{h-1}) \delta$.
\end{proof}

\begin{lemma} \label{spider-shrinking2}
	Let $G$ be a $K$-almost-regular graph with minimum degree $\delta$.
	Let $x$ be a vertex. Let $\cC$ be a family
	of at least $\alpha \delta^h$ distinct paths of length $h$ with one end $x$  and
	another end in a set $S$. Let $F$ be the subgraph of $G$ formed by taking the union of paths in $\cC$. For each $i\in [h]$ there exists a vertex $x_i$ and a balanced spider of height $i$ with center $x_i$
	and leaves in $S$ and has at least  $(\alpha/h K^{h-1}) \delta$ legs. Furthermore, if $i\neq h$, then $x_i\neq x$.
\end{lemma}
\begin{proof}
	Since $G$ has maximum degree at most $K\delta$, there are at most $(K\delta)^{h-i}$ distinct paths of length $h-i$ starting at $x$. So there is a path $Q$ 
	of length $h-i$ starting at $x$ and ending at some vertex $x_i$
	that is the initial segment of at least $|\cC|/(K\delta)^{h-i}\geq (\alpha/K^{h-i}) \delta^i$ members of  $\cC$.  If $i\neq h-i$, then $x_i\neq x$.
	Let $\cC_i$ denote the subfamily consisting of these members.
	Then $\{P-V(Q): P\in \cC'\}$ is a family of $|\cC'|$
	distinct paths of length $i$ each of which starts at $x_i$ and ends in $S$. By Lemma \ref{spider-shrinking}, $\cC'$ contains a subfamily of size at least $[(\alpha/K^{h-i})/iK^{i-1}]\delta
	\geq (\alpha/hK^{h-1})\delta$ which are vertex-disjoint outside $\{x_i\}$. The claim holds.
\end{proof}

An $s$-uniform hypergraph $\cF$ is called $s$-partite if there exists a partition of $V(\cF)$ into
$A_1,\dots, A_s$ such that each edge contains one vertex from each $A_i$. We call the $A_i$'s the parts. 

\begin{lemma} \label{min-degree}
	Let $\cF$ be an $s$-partite $s$-graph with parts $A_1,\dots, A_s$.
	Suppose that $|\cF|>\alpha |A_1|\cdots |A_s|$, where $\alpha>0$.
	Let $i\in [s]$.
	Then there exists a subgraph $\cF'$ such that $|\cF'|\geq (1/2) |\cF|$ and 
	for each
	$v\in A_i\cap V(\cF')$, $d_{\cF'} (v)>(\alpha/2) \prod_{j\in [s]\setminus \{i\}} |A_j|$.
\end{lemma}
\begin{proof}
	Let us call a vertex $v\in A_i$ {it $i$-bad} if its degree in the remaining graph is at most
	$(\alpha/2) \prod_{j\neq i} |A_j|$. As long as there exists an $i$-bad vertex for some $i\in [s]$,
	we remove all the edges containing that vertex. Let $\cF'$ be the remaining subgraph.
	Then at most $(\alpha/2) \prod_{i=1}^s |A_i|$ edges are removed in the process. So 
	$|\cF'|> (1/2)|\cF|$. Clearly $\cF'$ satisfies the degree requirement.
\end{proof}

Note that one could easily modify Lemma 1.8 to apply to all parts. But it suffices our purposes.
The following lemma provides one of the main ingredients of our proofs of the main results.

\begin{lemma} \label{balanced-spiders}
	Let $K\ge1$ and integers $k,s,t\geq 2$ be fixed. Then provided that $L$ is sufficiently large compared to $s,t,k$ and $K$, for any $\beta>0$ there exists $\delta_0$ such that the following holds. Suppose that $G$ is an $K_{s,t}^k$-free $K$-almost-regular graph $n$ vertices with minimum degree $\delta\ge\delta_0$. If $\ell_1,\dots, \ell_s$ are positive integers satisfying that $\forall i\in [s], \ell_i\geq k/2$ and that $\forall 1\leq i<j\leq s, \ell_i+\ell_j\geq k+1$, then the number of tuples $(w,v_1,\dots, v_s)$ such that there is an $(\ell_1,\dots, \ell_s)$-strong spider with center $w$ and leaf vector $(v_1,\dots, v_s)$ is at most
	$\beta n \delta^\ell$, where $\ell=\ell_1+\cdots +\ell_s$.
\end{lemma}
\begin{proof}
	For each vertex $w$ in $G$, let $\cH_w$ denote the family of tuples $(v_1,\dots, v_s)$ such that
	there is an $(\ell_1,\dots, \ell_s)$-strong spider with center $w$ and leaf vector $(v_1,\dots, v_s)$.
	Suppose for contradiction that there exists more than $\beta n \delta^\ell$ tuples
	$(w,v_1,\dots, v_s)$ such that there is an $(\ell_1,\dots, \ell_s)$-strong spider with center $w$
	and leaf vector $(v_1,\dots, v_s)$. Then 
	by the pigeonhole principle, there exists a vertex $w$ such that $|\cH_w|>\beta \delta^\ell$.
	Let us fix such a $w$. For each $(v_1,\dots, v_s)\in \cH_w$, by definition, we may fix a $(\ell_1,\dots, \ell_s)$-strong spider  $T(v_1,\dots, v_s)$ with leaf vector $(v_1,\dots, v_s)$. For each $i$, we
	call the path in $T(v_1,\dots, v_s)$ from $w$ to $v_i$ its $i$-th leg.
	
	Randomly and independently color vertices of $G$ with colors $1,\dots, s$ with each vertex
	receiving each color with probability $1/s$. For each $s$-tuple $(v_1,\dots, v_s)\in \cH_w$,
	we call it {\it good} if for each $i\in [s]$ all the vertices on the $i$-th leg of $T(v_1,\dots, v_s)$
	except $w$ are colored $i$. Since $T(v_1,\dots, v_s)-\{w\}$ has $\ell$ vertices, the probability 
	of $(v_1,\dots, v_s)$ being good is $1/s^\ell$.
	Hence, there exists a coloring $c$
	such that the following family 
	\[\cF_w=\{(v_1,\dots, v_s)\in \cH_w: (v_1,\dots, v_s) \mbox { is good} \}\]
	
	satisfies 
	\begin{equation} \label{Fw-bound1}
	|\cF_w|\geq |\cH_w|/s^\ell > (\beta/s^\ell) \delta^\ell.
	\end{equation}
	
	Let us fix such a coloring $c$. For each $i\in [s]$, let 
	\[A_i=\{v\in V(\cF_w): c(v)=i\}.\] 
	Then $\cF_w$ is an $s$-partite $s$-graph with parts $A_1,\dots, A_s$.
	By our assumption, for each $i\in [s]$ and each $v\in A_i$ there is an $(\ell_1,\dots, \ell_s)$-strong spider with center $w$ where $v$ plays the role of the $i$-th vertex in the leaf vector. Furthermore,
	all the vertices on the $i$-th leg, except $w$, are colored $i$ under $c$. Since $G$ has maximum degree at most $K\delta$, we have
	\begin{equation} \label{part-sizes}
	\forall i\in [s], \, |A_i|\leq (K\delta)^{\ell_i}.
	\end{equation}
	Let $\alpha=\frac{\beta}{s^\ell K^\ell}$. For each $i\in [s]$, let $\alpha_i=\frac{\beta}{s^\ell K^{\ell-\ell_i}}$.
	By \eqref{Fw-bound1} and \eqref{part-sizes}, we have
	
	\begin{equation} \label{Fw-bound2}
	|\cF_w|> \alpha_ |A_1|\dots |A_s| \quad \mbox{ and } \quad 
	\forall i\in [s], |A_i|\geq |\cF_w|/\prod_{j\neq i} |A_j| \geq \alpha_i \delta^{\ell_i}.
	\end{equation}
	
	Now, we may assume without loss of generality that $\ell_1\leq \ell_2\leq \cdots\leq \ell_s$.
	First, let us observe that if $\ell_1=\ell_2=\cdots=\ell_s=k$, then we may take any $(k,\dots, k)$-strong tuple $(v_1,\cdots, v_s)$. By the definition of strong tuples, there are at least $f(k,L)$ internally
	vertex-disjoint spiders with leaf vector $(v_1,\dots, v_s)$.  It is easy to see that the union of any
	$t$ of these spiders form a copy of   $K^{k}_{s,t}$, contradicting $G$ being $K^{k}_{s,t}$-free.
	Hence, we may assume that $\ell_1<k$. For each $i\in [s]$, let $m_i=k-\ell_i$. 
	By our assumption, $\forall i\in [s], \ell_i\geq k/2$ and $\forall i,j\in [s], \ell_i+\ell_j>k$.
	This implies that 
	\[m_1\leq \ell_1 \mbox{ and  } \forall 2\leq i\leq s, m_i<\ell_i.\]
	
	Let $q=\max \{i: \ell_i<k\}$. Then $1\leq q\leq s$.
	By Lemma \ref{min-degree}, $\cF_w$ contains a subgraph $\cF_1$ such that
	\begin{equation} \label{F1-bound}
	|\cF_1|> (1/2)|\cF_w| >  (\alpha /2)|A_1|\cdots |A_s|.
	\end{equation}
	and 
	\begin{equation} \label{A1-min-degree}
	\forall v\in A'_1:=A_1\cap V(\cF_1),\,\, d_{\cF_1}(v)\geq (\alpha /2) \prod_{j\neq 1}  |A_j|.
	\end{equation}  
	By  \eqref{F1-bound} and \eqref{Fw-bound2}, we have
	\begin{equation} \label{A1-prime-bound}
	|A'_1|\geq (\alpha/2)|A_1| \geq (1/2) \alpha \alpha_1 \delta^{\ell_1}.
	\end{equation}
	
	For each $v\in A'_1$ there is an edge of $\cF_1$ containing it, which in particular, by
	our earlier discussion, implies that
	there is a path $P_v$ of length $\ell_1$ from $w$ to $v$, all of which except $w$
	are colored $i$ by $c$. Let 
	\[\beta_1= \frac{(1/2)\alpha \alpha_1}{\ell_1 K^{\ell_1-1}}.\]
	By Lemma \ref{spider-shrinking2}, there exists a vertex $z_1$ and a balanced spider $S_1$
	of height $m_1$
	with center at $z_1$ and leaf set $B_1\subseteq A'_1$ such that
	\[\beta_1 \delta\leq |B_1|\leq \delta.\]
	
	Note that if $m_1=\ell_1$, then $z_1=w$. If  $m_1<\ell_1$, then $z_1\neq w$.
	Also, all the vertices in $S_1$, except possibly $w$, have color $1$ in $c$.
	Since $B_1\subseteq A'_1$, by \eqref{A1-min-degree} 
	\begin{equation} \label{B1-min-degree}
	\forall v\in B_1, \,\, d_{\cF_1}(v)\geq (\alpha /2) \prod_{j\neq 1}  |A_j|.
	\end{equation}

	In general, let $1\leq i\leq q-1$ and suppose we have defined $\cF_1,\dots,\cF_i$ and 
	$B_1,\dots, B_i$ such that 
	\[ |\cF_i|\geq (\alpha/2^i)|B_1|\cdots |B_{i-1}||A_i|\cdots |A_s| \]
	and 
	\[ \beta_i \delta\leq |B_i|\leq \delta, \mbox{ where }
	\beta_i=\frac{(1/2^i) \alpha \alpha_i}{\ell_i K^{\ell_i-1}}.\]
	Furthermore, suppose
	\begin{equation} \label{Bi-min-degree}
	\forall v\in B_i, \, \, d_{\cF_i}(v)\geq (\alpha/2^i) |B_1|\cdots |B_{i-1}||A_{i+1}|\cdots |A_s|.
	\end{equation}

	Also, suppose that there are distinct vertices $z_1,\dots, z_i$ such that
	for each $j\in [i]$, there is a balanced spider $S_j$ of height $m_j$ with center $z_j$
	and leaf set $B_j$, all of whose vertices except possibly $w$ lie in color class $j$ of $c$. 
	Also, suppose that $z_2,\dots, z_i\neq w$ and $z_1=w$ if and only if $\ell_1=m_1$.
	Now, let $\cH_{i+1}$ be the subgraph of $\cF_i$ induced by $B_1\cup \cdots \cup B_i \cup A_{i+1}\cup\dots \cup A_s$.
	By \eqref{Bi-min-degree}, 
	\begin{equation} \label{Hplus-bound}
	|\cH_{i+1}|\geq (\alpha/2^i) |B_1|\cdots |B_i||A_{i+1}|\cdots |A_s|.
	\end{equation}
	
	By Lemma \ref{min-degree}, $\cH_{i+1}$ contains a subgraph $\cF_{i+1}$ such that
	\begin{equation} \label{Fplus-bound}
	|\cF_{i+1}|\geq (1/2)|\cH_{i+1}| \geq (\alpha/2^{i+1})|B_1|\cdots |B_i|
	|A_{i+1}|\cdots |A_s|.
	\end{equation}
	and 
	\begin{equation} \label{A'plus-min-degree}
	\forall v\in A'_{i+1}:=A_{i+1}\cap V(\cF_{i+1}), \,\, d_{\cF_{i+1}}(v)\geq (\alpha/2^i) |B_1|
	\cdots |B_i| |A_{i+2}|\cdots |A_s|.
	\end{equation}  
	By \eqref{A'plus-min-degree} and \eqref{Fw-bound2} we have
	
	\begin{equation} \label{A2-prime-bound}
	|A'_{i+1}|\geq (\alpha/2^{i+1})|A_{i+1}|\geq (1/2^{i+1})\alpha \alpha_{i+1} \delta^{\ell_{i+1}}.
	\end{equation}
	
	As before, for each $v\in A'_{i+1}$ 
	there is a path $P_v$ of length $\ell_{i+1}$ from $w$ to $v$, all of which except $w$ have color $i+1$ in $c$. Let 
	\[\beta_{i+1}= \frac{(1/2^{i+1})\alpha \alpha_{i+1}}{\ell_{i+1} K^{\ell_{i+1}-1}}.\]
	
	By Lemma \ref{spider-shrinking2}, there exists a vertex $z_{i+1}$ and a balanced spider $S_{i+1}$
	of height $m_{i+1}$
	with center $z_{i+1}$ and leaf set $B_{i+1}\subseteq A'_{i+1}$ such that
	\[\beta_{i+1} \delta\leq |B_{i+1}|\leq \delta.\]
	
	Furthermore, since $m_{i+1}<\ell_{i+1}$, we have $z_{i+1}\neq w$. Also, all the vertices in $S_{i+1}$
	lie in color class $i+1$ of $c$.
	Since $B_{i+1}\subseteq A'_{i+1}$, by \eqref{A'plus-min-degree} 
	\[\forall v\in B_{i+1}, \, \, d_{\cF_{i+1}}(v)\geq (\alpha/2^{i+1}) |B_1|\cdots |B_i|
	|A_{i+2}|\cdots |A_s|.\]
	
	This allows to define $\cF_1,\dots, \cF_q$, $B_1,\dots, B_q$, and $z_1,\dots, z_q$. Now, we claim that we can find a copy of $K_{s,t}^k$ in $G$, which would give us a contradiction. To find such a copy, we consider two subcases. 
	
	\medskip
	
	{\bf Case 1.} $q=s$.
	
	\medskip
	By our assumption, $\cF_s$ is an $s$-partite $s$-graph with parts $B_1,\dots, B_s$, where
	\[|\cF_s|\geq (\alpha/2^s) |B_1|\cdots |B_s|\] and
	\[\forall i\in [s], \,\, \beta_i \delta\leq |B_i|\leq \delta, \mbox{ where } \beta_i= \frac{(1/2^s) \alpha \alpha_s}{
		\ell_s K^{\ell_s-1}}.\]
	
	Let $\cM$ be a maximum matching in $\cF_s$. Then the maximality of $\cM$ implies that every edge of $\cF_s$ contains some vertex in $V(\cM)$. On the other hand, since $\cF_s$ is $s$-partite and each part has size at most $\delta$, each vertex is contained in at most $\delta^{s-1}$ edges. Hence
	\[|\cF_s|\leq |V(\cM)|\cdot \delta^{s-1}=s|\cM|\delta^{s-1}.\]
	Hence by the above lower bounds on $|\cF_s|$ and $|B_1|,\dots, |B_s|$, we have
	\[|\cM|\geq |\cF_s|/s \delta^{s-1} \geq (\alpha \beta_1\cdots \beta_s/2^s) \delta\gg t,\]
	for sufficiently $\delta$ (as $\delta\ge\delta_0$).  
	Let $\cM'$ be a set of $t$ edges in $\cM$. Suppose $\cM'=\{e_1,\dots, e_t\}$.
	For each $i\in [t]$, suppose $e_i=(v_1^i, v_2^i,\dots, v_s^i)$, where $\forall j\in [s],
	v_j^i\in B_j$.  For each $j\in [s]$, let $Z_j$ be the sub-spider of $S_j$ 
	obtained by keeping only the $t$ paths from $z_j$ to $V(\cM')\cap B_j$.
	Since vertices in $Z_1-\{w\}$ have color $1$ and
	for each $2\leq j\leq s$, vertices in $Z_j$ have color $j$, $Z_1\dots, Z_t$ are vertex-disjoint.
	
	By the definition of $\cF_s\subseteq H_w$, 
	for each $i\in [t]$, $(v_1^i,\dots, v_s^i)$ is a strong $(\ell_1,\dots, \ell_s)$-tuple
	and hence there are $f(k,L)$ internally vertex-disjoint spiders with leaf vector $(v^i_1,\dots,
	v^i_s)$ and length vector $(\ell_1,\dots, \ell_s)$. Since $f(k,L)\gg  |V(K_{s,t}^k)|$, 
	we can greedily find $t$ vertex disjoint spiders $T_1,\dots T_t$
	such that for each $i\in [t]$, $T_i$ has leaf vector $(v^i_1,\dots, v^i_s)$
	and length vector $(\ell_1,\dots,\ell_s)$ and that $V(T_i)\setminus\{v_1^i, v_2^i,\dots, v_s^i\}$
	is disjoint from $\bigcup_{j=1}^s V(Z_j)$.  
	Now $(\bigcup_{i=1}^t T_i)\cup (\bigcup_{j=1}^s Z_j)$
	forms a copy of $K^{k}_{s,t}$, contradicting $G$ being $K^{k}_{s,t}$-free.
	
	\medskip
	
	{\bf Case 2.} $q<s$.
	
	\medskip
	
	Since $|\cF_s|\geq (\alpha/2^s)|B_1|\cdots |B_s|$, by averaging, there exists a tuple
	$(z_{q+1},\cdots, z_s)\in B_{q+1}\times \cdots \times B_s$ that is contained in at least
	$(\alpha/2^s) |B_1|\cdots |B_q|$ of the edges of $\cF_s$. Let 
	\[\cF^*=\{e\setminus \{z_{q+1},\dots, z_s\}: \{z_{q+1},\dots, z_s\}\subseteq e \in \cF_s\}.\]
	As in Case 1, for sufficiently large $n$, $\cF^*$ contains a matching $\cM=\{e_1,\dots, e_t\}$ of
	size $t$.
	
	For each $i\in [t]$, suppose $e_i=(v_1^i, v_2^i,\dots, v_q^i)$, where $\forall j\in [q],
	v_j^i\in B_j$.  For each $j\in [q]$, let $Z_j$ be the sub-spider of $S_j$ 
	obtained by keeping only the $t$ paths from $z_j$ to $V(\cM)\cap B_j$.
	Since vertices in $Z_1-\{w\}$ have color $1$ and
	for each $2\leq j\leq s$, vertices in $Z_j$ have color $j$, $Z_1\dots, Z_t$ are vertex-disjoint.
	
	By definition, 
	for each $i\in [t]$, $(v_1^i,\dots, v_q^i,z_{q+1},\dots, z_s)$ is a strong $(\ell_1,\dots, \ell_s)$-tuple
	and hence there are $f(k,L)$ internally vertex-disjoint spiders with leaf vector $(v^i_1,\dots,
	v^i_s)$ and length vector $(\ell_1,\dots, \ell_s)$. Since $f(k,L)\gg |V(K_{s,t}^k)|$, 
	we can greedily find $t$ spiders $T_1,\dots T_t$
	such that for each $i\in [t]$, $T_i$ has leaf vector $(v^i_1,\dots, v^i_q, z_{q+1},\dots, z_s)$
	and length vector $(\ell_1,\dots,\ell_s)$ and that $V(T_i)\setminus \{z_{q+1},\dots, z_s\}$ are
	pairwise disjoint over different $i$ and that $V(T_i)\setminus\{v_1^i, v_2^i,\dots, v_s^i\}$
	is disjoint from $\bigcup_{j=1}^q V(Z_j)$ for each $i\in [t]$.  Now $(\bigcup_{i=1}^t T_i)\cup (\bigcup_{j=1}^s Z_j)$
	forms a copy of $K^{k}_{s,t}$, contradicting $G$ being $K_{s,t}^k$-free.
\end{proof}

From Lemma \ref{balanced-spiders}, we immediately obtain the following.

\begin{corollary} \label{general-case-cor}
	Let $K\ge1$ and integers $k,s,t\geq 2$ be fixed. Then provided that $L$ is sufficiently large compared to $s,t,k$ and $K$, for any $\beta>0$ there exists $\delta_0$ such that the following holds. Suppose that $G$ is an $K_{s,t}^k$-free $K$-almost-regular graph $n$ vertices with minimum degree $\delta\ge\delta_0$. Suppose $\ell_1,\dots, \ell_s$ are positive integers satisfying that $\forall i\in [s], \ell_i\geq k/2$ and that $\forall 1\leq i<j\leq s, \ell_i+\ell_j\geq k+1$. Let $\ell=\ell_1+\dots +\ell_s$.
	Let $\cF$ denote the family of all the balanced $s$-legged spiders in $G$ of height $k$ that
	contain a $(\ell_1,\dots, \ell_s)$-strong sub-spider but contain no critical path of length at most $k$. Then $|\cF| \leq [f(k,L)]^s\cdot \beta n \delta^\ell$.
\end{corollary}
\begin{proof}
Let $F\in \cF$. By Definition \ref{heavy-light-path-definition}, since $F$ contains no critical paths of length at most $k$, it also does not contain any heavy paths of length at most $k$. By Lemma \ref{balanced-spiders}, there are at most $\beta n\delta^\ell$ tuples $(w,\ell_1,\dots, \ell_s)$ such
that there is a member of $\cF$ that has $w$ as the center and $(\ell_1,\dots, \ell_s)$ as the leaf vector.
Each such tuple corresponds to at most $[f(k,L)]^s$ different members of $\cF$, since for each $i$,
there are at most $f(k,L)$ light paths of length $k$ in $G$ between $w$ and $v_i$.
\end{proof}


\subsection{ Building subdivisions using strong spiders: the $(1,k,\dots, k)$-case}

In this section, we prove a second crucial ingredient (Lemma \ref{1kk-spiders} below) which complements Lemma \ref{balanced-spiders}. First we need a lemma (Lemma \ref{spider} below), 
which is  a slight adaption of \cite{JS} Lemma 2.4. Given a $u,w$-path $P$ in a graph and  vertices $x,y$ on $P$, we let $P[x,y]$ denote the portion of $P$ from $x$ to $y$.

\begin{lemma} \label{spider1}
Let $m,k$ be positive integers.
Let $A_0,A_1,\dots, A_k$ be disjoint sets of vertices, where $A_0=\{z\}$ and $|A_k|\geq m^k$.
For each $w\in A_k$, let $P_w$ be a fixed $z,w$-path of length of $k$ that contains exactly one vertex
of each $A_i$. Then there exists a vertex $x\in A_{k-j}$ for some $j\in [k]$ and $m$ vertices
$w_1,\dots, w_m$ in $A_k$ such that $\{P_{w_i}[x,w_i]: i\in [m]\}$ is a family of paths of length $j$
every two of which share only $x$ as a common vertex.
\end{lemma}
\begin{proof}
	We prove the statement by induction on $k$. The case of $k=1$ is trivial. Assume that $k\ge2$ and the statement holds for $k-1$.  For each $w\in W$, let $f(w)$ denote the vertex on $P_w$ that precedes $w$. Then we may view $S:=\{f(w): w\in W\}$ as a multi-set of size $|W|$. If some vertex $x$ in $S$ has multiplicity $m$ in $S$, then there exist $w_1,\dots, w_m$ such that $f(w_1)=f(w_2)=\dots=
	f(w_m)=x$ and the claim holds with $j=1$. Hence, we may assume that each vertex in $S$ is the image of fewer than $m$ vertices in $A_k$ under $f$. In particular, we have
	\[|S|\geq |A_k|/m=m^{k-1}.\]
	For each $y\in S$, let $h(y)$ be an arbitrary pre-image of $y$ under $f$. Then $h$ is an injection from $S$ to $A_k$.  For each $y\in S$ let $Q_y=P_{h(y)}[z,y]$. Since $|S|\geq m^{k-1}$,
by the induction hypothesis, there exist some vertex $x\in A_{k-1-j}$ for some $j\in [k-1]$ and
$m$ vertices $y_1,\dots y_m$  in $S$ such that $\{Q_{y_i}[x,y_i]: i\in [m]\}$ is a family of
paths of length $j$ every two of which share only $x$ as a common vertex. For $i\in [m]$,
$Q_{y_i}[x,y_i]\cup y_i h(y_i)=P_{h(y_i)}[x, h(y_i)]$ form a family of paths that satisfy the statement.
\end{proof}

\begin{lemma}\label{spider}
	Let $m$ and $k$ be integers.  Let $z$ be a vertex and $W$ is a set not containing $z$.
For each $w\in W$, let $P_w$ be a $z,w$-path of length $k$ and let $\cF=\{P_w: w\in W\}$.
	 If $|W|\ge (mk)^k$, then for some $j\in [k]$ there exist a vertex $x$ and 
	  $m$ vertices $w_1,\dots,  w_m$ in $W$ such that $\{P_{w_i}[x,w_i]: i\in [m]\}$ is a family of paths of length $j$  every two of which share only $x$ as a common vertex.
\end{lemma}
\begin{proof} Let us randomly and independently color the vertices in $\bigcup_{w\in W} V(P_w)-\{z\}$ using $1,\dots, k$ with each color chosen with probability $1/k$.
Let us call a $P_w\in \cF$ {\it good} if for each $i\in [k]$ the vertex on $P_w$ at distance $i$ from $z$ is colored $i$. The probability of any $P_w$ being good is $(1/k)^k$. Hence there exists a coloring for which the number of good $P_w$'s is at least $(mk)^k/k^k=m^k$. Now the claim follows immediately from Lemma \ref{spider1}.
\end{proof}

\begin{lemma} \label{1kk-spiders}
	Let $K\ge1$ and integers $k,s,t\geq 2$ be fixed. Then provided that $L$ is sufficiently large compared to $s,t,k$ and $K$, for any $\gamma>0$ there exist $n_0,C>0$ such that the following holds. Suppose that $G$ is an $K_{s,t}^k$-free $K$-almost-regular graph $n\ge n_0$ vertices with minimum degree $\delta\ge Cn^{\frac1k-\frac1{sk}}$. 
	Let $\cF$ denote the family of all the balanced $s$-legged spiders of height $k$ in $G$ that contain  a
	$(1,k,\dots, k)$-strong $s$-legged spider but do not contain any critical paths of length at most $k$ or any $(\ell_1,\dots, \ell_s)$-strong sub-spider for any 
	$(\ell_1,\dots, \ell_s)\neq (1,k,\dots, k)$. Then
	$|\cF|\leq \gamma n\delta^{sk}$.
\end{lemma}
\begin{proof}	
Suppose to the contrary that $|\cF|\ge \gamma n\delta^{ks}$. We derive a contradiction.
First we do some cleaning.
Let $c>0$ such that $(skK)^{sk}c=\frac{\gamma}4$ and let 
\[\partial(\cF)= \{T: T \mbox{ is a tree on at most $ks$ vertices and } \exists F\in \cF, E(T)\subseteq E(F)\}\]
	As long as there exists  $T\in\partial(F)$ such that there are fewer than $c\delta\cdot (K\delta)^{sk-e(T)-1}=c\delta (K\delta)^{sk-|T|}$ members of $\cF$  that contain $T$, we delete all these members from $\cF$; otherwise, terminate. Let $\cF'$ denote the remaining subfamily of $\cF$.
	
For each $j\in [sk]$ let $\partial_j(\cF)=\{T\in \partial(\cF): |T|=j\}$. Let $\cT_j$ denote the set of
all labelled trees on $[j]$. By Cayley's formula $|\cT_j|\leq j^{j-2}<j^j$. For each member $T\in \cT_j$, there are
at most $n\cdot (K\delta)^{j-1}$ copies of $T$ in $G$, since $G$ has maximum degree at most $K\delta$. Hence $|\partial_j(\cF)|\leq j^j\cdot n(K\delta)^{j-1}$. On the other hand, for each 
$T\in \partial_j(\cF)$,
by rule, we  have deleted fewer than $c\delta\cdot (K\delta)^{sk-j}$ members from $\cF$ that contain $T$. Thus the total number of members we have deleted from $\cF$ 
is less than $\sum_{j=1}^{sk} j^j\cdot n(K\delta)^{j-1}\cdot c\delta (K\delta)^{sk-j}
\le(sk)^{sk}cn\delta(K\delta)^{sk-1}\le\frac14\gamma n\delta^{ks}$. 
Hence 
	\[|\cF'|\ge\gamma n\delta^{ks}-\frac14\gamma n\delta^{ks}=\frac34\gamma n\delta^{ks},\]
and by the definition of $\cF'$
\begin{equation} \label{high-codegree}
\forall T\in \partial(\cF')\, \mbox{ there are at least $c\delta (K\delta)^{sk-|T|}$ members
of $\cF'$ that contain $T$}.
\end{equation}

Given an $(s-1)$-tuple $\va=(a_1,\dots, a_{s-1})$ of vertices in $G$, let $\cL_{\va}$ denote the subfamily of members of $\cF'$ that contain $a_1,\dots, a_{s-1}$ as leaves. 
For each $F\in \cL_{\va}$, let $w(F)$ denote the center of $F$ and let $u(F)$ denote
the neighbor of $w(F)$  on the path from $w(F)$ to the remaining leaf $z$.
For each $F\in \cL_{\va}$, let $F|_{\va}$ denote the subtree obtained from $F$ by replacing the
$w(F),z$- path in it with $w(F)u(F)$.
If $F|_{\va}$ is a $(1,k,\dots, k)$-strong spider with center $w(F)$ and leaf vector $(u(F),a_1,\dots, a_{s-1})$ then we say that $F_{\va}$ is {\it good}.
For each $(s-1)$-tuple $\va$  let 
\[\cF_{\va} =\{F\in \cF': \cF|_{\va} \mbox{ is good} \},\]
and let 
\[H_{\va}=\{w(F)u(F): F\in \cF_{\va}\}.\]

Furthermore, let 
\[W_{\vec{a}} =\{w(F): F\in \cF_{\va}\}  
\quad \mbox{ and } \quad
U_{\vec{a}} =\{u(F): F\in \cF_{\va}\}.  \] 

Since $G$ is bipartite, we have $W_{\va}\cap U_{\va}=\emptyset$.
Hence $H_{\va}$ is bipartite with parts $W_{\va}$ and $U_{\va}$.
Observe that by definition,
\begin{equation} \label{strong-1kk-spider}
\forall u\in U_{\va}, \mbox{ there is a $(1,k,\dots, k)$-strong spider in $G$ with 
leaf vector $(u,a_1,\dots, a_{s-1})$}.
\end{equation} 

\medskip

{\bf Claim 1.} Let $\va$ be a $(s-1)$-tuple such that $\cF_{\va}\neq\emptyset$.
Let $uw\in H_{\va}$, where $u\in U_{\va}$ and $w\in W_{\va}$. Then the  number of
members of $\cF'$ containing $uw$ is at least $c\delta(K\delta)^{k-2}$ and at most
$[f(k,L)]^{s-1}\cdot (K\delta)^{k-2}$. The number of members of $\cF'$ containing $w$ is at
least $c\delta(K\delta)^{k-1}$.

\medskip

{\it Proof of Clam 1.} By definition, there is a member $F\in \cF_{\va}$ such that $w(F)=w$ and $u(F)=u$. Let $F^*=F|_{\va}$. Then $F^*\in \partial(\cF')$.  Since $|F^*|=(s-1)k+2$,
by \eqref{high-codegree}, there are at least $c\delta(K\delta)^{k-2}$ members of $\cF'$ that contain $F^*$ and hence contain $uw$. To upper bound the  number of members of $\cF'$ that contain $uw$,
note that there are at most $[f(k,L)]^{s-1}$ ways to pick the paths from $a_i$ to $w$ for $i\in [s-1]$
and at most $(K\delta)^{k-2}$ ways to grow such a member past $u$.  Now, let $S$ be obtained
from $F^*$ by deleting $u$. Then $S\in \partial(\cF')$. Since $|S|=(s-1)k+1$, by
\eqref{high-codegree}, there are at least $c\delta(K\delta)^{k-1}$ members of $\cF'$ that
contain $S$ and hence contain $w$.\qed

\medskip

{\bf Claim 2.} For each $(s-1)$-tuple $\va$ for which $\cF_{\va}\neq\emptyset$, we have
$|\cF_{\va}|\geq e(H_{\va})\cdot c\delta (K\delta)^{k-2}$ and 
$e(H_{\va})\geq c'\delta\cdot|W_{\va}|$, where $c'=cK/[f(k,L)]^{s-1}$.

\medskip

{\it Proof of Claim 2.} By Claim 1, for each $wu\in H_{\va}$ there are at least 
$c\delta(K\delta)^{k-2}$ members $F$ of $\cF'$ that contain $wu$. Since different $wu$'s clearly
give rise to different $F$'s, the first part of the claim follows.

Now, let $w\in W_{\va}$. By Claim 1, there are at least $c\delta(K\delta)^{k-1}$ members
of $\cF'$ that contain $w$. Each such member contains $wu$ for some edge $wu\in H_{\va}$.
On the other hand, for each such fixed $wu$, by Claim 1, there are at most $[f(k,L]^{s-1} (K\delta)^{k-2}$ members of $\cF'$ that contain it. This implies that
\[ d_{H_{\va}}(w)\geq  \frac{c\delta(K\delta)^{k-1}}{ [f(k,L)]^{s-1}\cdot (K\delta)^{k-2}} =c'\delta.\] 
So $e(H_{\va})\geq c'\delta\cdot|W_{\va}|$.
\qed

\medskip

For any $(s-1)$-tuple $\vec{a}=(a_1,\ldots,a_{s-1})$, let

\[U_{\va}^{+}=\{u\in U_{\va}: d_{H_{\va}}(u)\ge 2kt\} \quad \mbox{and} \quad 
U_{\va}^{-}:=\{u\in U_{\va}: d_{H_{\va}}(u)<2kt.\}.\]
Let
\[\cF_{\va}^+=\{F\in \cF_{\va}: u(F)\in U^+_{\va}\}, \quad \mbox{ and }  \quad 
\cF_{\va}^-=\{F\in \cF_{\va}: u(F)\in U^-_{\va}\}.\]

\medskip

{\bf Claim 3.} For every $(s-1)$-tuple $\va$ we have
$e(H_{\va}[U^+_{\va},W_{\va}])\le 2kt|W_{\va}|$.

\medskip

{\it Proof of Claim 3.} Let $\va$ be given. For convenience, let $U^+=U^+_{\va}$
and $W=W_{\va}$.
Suppose that
$e(H_{\va}[U^+,W])>2kt|W|$. Then this, together with the definition of $U^+_{\va}$, implies that the average degree of $H_{\va}[ U^{+},W]$ is at least $2kt$.  By a well-known fact, $H_{\va}[U^+,W]$ contains a subgraph $H'$ with minimum degree at least $kt$.
In $H'$, we can greedily build a $t$-legged spider $T$ of height $k-1$ with leaves lying in $U$. Let $x$ be its center and $u_1,\ldots,u_t$ be its leaves.  By \eqref{strong-1kk-spider}, $(u_i,a_1,\ldots,a_{s-1})$ is $(1,k,\ldots,k)$-strong for every $i\in[t]$. Thus using strong-ness one can greedily find $t$ internally disjoint  balanced spiders of height $k$ with leaf vector $(x,a_1,\ldots,a_{s-1})$. The union of these $t$ spiders forms a copy of $K_{s,t}^k$, contradicting $G$ being $K_{s,t}^k$-free. \qed

\medskip

By Claims 1 and 3, we have
\begin{equation} \label{Fplus}
|\cF_{\va}^+|\le e(H_{\va}[U^+_{\va},W_{\va}])\cdot[f(k,L)]^{s-1}(K\delta)^{k-1}\le [2kt[f(k,L)]^{s-1}K^{k-1}]\cdot|W_{\va}|\cdot\delta^{k-1}.
\end{equation}
On the other hand, by Claims 2 we have
$$|\cF_{\va}|\ge e(H_{\va})\cdot c\delta(K\delta)^{k-2}\ge c'\delta|W_{\va}|\cdot c\delta(K\delta)^{k-2} = c'cK^{k-2}\cdot|W_{\va}|\cdot\delta^k.$$
As $\delta\ge Cn^{\frac1k-\frac1{sk}}$ and $n\ge n_0$ is sufficiently large, this, together with \eqref{Fplus} yields that
\[|\cF_{\va}^+|\le\frac12|\cF_{\va}|.\]
Thus $|\cF_{\va}^-|=|\cF_{\va}|-|\cF_{\va}^+|\ge\frac12|\cF_{\va}|$.
Since $\cF'=\cup_{\va}\cF_{\va}$, we have that $\sum_{\va}|\cF_{\va}|\ge|\cF'|\ge\frac34\gamma n\delta^{sk}$. It follows that
\begin{equation*}
\sum_{\va}|\cF_{\va}^-|\ge\frac12\sum_{\va}|\cF_{\va}|\ge\frac38\gamma n\delta^{ks}\ge\frac{3\gamma C^{sk}}{8}n^{s}.
\end{equation*}
By averaging, there exists an $(s-1)$-tuple $\va$ such that $|\cF_{\va}^-|\ge C_1n$, for some constant $C_1$ that can be made arbitrarily large by taking $C$ to be sufficiently large. By averaging again, there exists some $z$ such that the number of spiders in $\cF_{\va}^-$ with leaf vector $(\va,z)$ is at least $C_1$. Fix such a vertex $z$ and let 
\[\cF_{\va,z}=\{F\in \cF_{\va}^-: F \mbox{  has leaf vector }(\va, z) \}.\]
Let 
\[W_{\va,z}= \{w(F): F\in \cF_{\va,z}\}.\]
Note that for each $w\in W_{\va,z}$, since members of $\cF_{\va,z}$ by requirements contain no
critical paths of length at most $k$ and hence no heavy paths of length at most $k$, the number of these members that have $w$ as the center
and $(\va,z)$ as leaf vector is at most $[f(k,L)]^s$. Hence
\[|W_{\va,z}|\geq \frac{|\cF_{\va,z}|}{[f(k,L)]^s}\geq \frac{C_1}{[f(k,L)]^s}.\]
By choosing $C$ to be sufficiently large (which makes $C_1$ sufficiently large) we can ensure
\[|W_{\va,z}|\geq [(sk)^k \cdot f(k,L)\cdot k]^k. \]

{\bf Claim 4.} Some member of $\cF_{\va,z}$ contains a $(j,k,\dots,k)$-strong sub-spider for some
$2\leq j\leq k$.

\medskip 

{\it Proof of Claim 4.} For each $F\in \cF_{\va,z}$, let $P_F$ denote the $z,w(F)$-path in $F$.
For each $w\in W_{\va,z}$,
by the definition of $W_{\va,z}$ there exists some $F\in \cF_{\va,z}$ such that
$w(F)=w$. Fix such an $F$ and let $P_w=P_F$. Let 
\[\cC=\{P_w: w\in W_{\va,z}\}.\]
Let $m=(sk)^k \cdot f(k,L)$. Since $|W_{\va,z}|\geq (mk)^k$, by Lemma \ref{spider}, for some $j\in [k]$ there exist a vertex $x$ and 
$m$ vertices $w_1,\dots,  w_m\in W_{\va,z}$ such that $J:=\bigcup_{i\in[m]}P_{w_i}[x,w_i]$ is a spider with center $x$ and height $j$. If $j=1$, then $J$ is a star of size at least $m=(sk)^{sk}\cdot f(k,L)\gg 2kt$ 
in $H_{\va}[U^-_{\va},W_{\va} ]$ with the center $x\in U^-_{\va}$, contradicting the definition of
$U^-_{\va}$. Hence $j\geq 2$.

It remains to show that $(x,\va)$ is $(j,k,\ldots,k)$-strong. 
As $P_{w_1}\in\cC$, by the definition of $\cC$, there exists some  $F\in\cF_{\va,z}$ such that $P_{w_1}=P_F$. In particular, $w(F)=w_1$. Let $F'$ be the sub-spider obtained from $F$ by replacing $P_F$ with $P_F[x,w_1]$.
Then $F'$ has leaf vector $(x,\va)$ and length vector $(j,k,\ldots,k)$. If one can prove that the tuple $(x,\va)$ is $(j,k,\ldots,k)$-strong, then by definition, $F'$ is $(j,k,\dots,k)$-strong and thus $F$ contains a $(j,k,\dots,k)$-strong sub-spider, which would prove the claim.

Next, we show that indeed $(x,\va)$ is $(j,k,\ldots,k)$-strong. Let $$q=(sk)^{k-j}\cdot f(k,L).$$
By the definition of $(j,k,\ldots,k)$-strong-ness, we need to show there exist $q$ internally disjoint spiders with leaf vector $(x,\va)$ and length vector $(j,k,\dots, k)$. 
For each $i\in [q]$, let $u_i$ be the vertex on $P_{w_i}[x,w_i]$ that precedes $w_i$, and let $P_i=P_{w_i}[x,u_i]$ for short.   
Note that for each $i\in [q]$, 
$u_i\in U_{\va}^-\subseteq U_{\va}$ and hence in particular $(u_i,\va)$ is $(1,k,\dots, k)$-strong.
We will greedily find $q$ spiders $T_1,\ldots, T_q$ with length vector $(1,k,\ldots,k)$, satisfying that every $T_i$ has leaf vector $(u_i,\va)$ and that $T_1\cup P_1,\ldots,T_q\cup P_q$ are $q$ internally disjoint spiders with leaf vector $(x,\va)$ and length vector $(j,k,\ldots,k)$.
Since $(u_1,\va)$ is $(1,k,\ldots,k)$-strong, there are at least $(sk)^{k-1}\cdot f(k,L)=(sk)^{j-1}q$ internally disjoint spiders with leaf vector $(u_1,\va)$ and length vector $(1,k,\dots, k)$. As $|V(\cup_{i=1}^q P_i)\setminus\{u_1\}|=(j-1)q<(sk)^{j-1}q$, there exists one such spider $T_1$ such that $V(T_1)\cap V(\cup_{i=1}^q P_i)=\{u_1 \}$. In general, suppose that for some $p\le q$ we have found $T_1,\ldots,T_{p-1}$ such that for each $i\in[p-1]$ $T_i$ is a spider with leaf vector $(u_i,\va)$ and length vector $(1,k,\dots, k)$ and $V(T_i)\cap V(\cup_{i=1}^q P_i)=\{u_i \}$ and that $V(T_i)\setminus \{a_2,\ldots,a_{s}\}$ are disjoint over all $i\in[p-1]$. As $(u_p,\va)$ is $(1,k,\dots, k)$-strong  and $|V(\cup_{i=1}^q P_i)\setminus\{u_p\}|=(j-1)q$, there are at least $(sk)^{j-1}q-(j-1)q\ge (sk-1)q$ internally disjoint spiders $T_p$ with leaf vector $(u_p,\va)$ and length vector $(1,k,\dots, k)$, such that $V(T_p)\cap V(\cup_{i=1}^q P_i)=\{u_p \}$. Since the size of $X:=\cup_{i=1}^{p-1}(V(T_i)\setminus \{a_2,\ldots,a_{s}\})$ is $(sk-s-k+2)(p-1)\le(sk-s)q$, among these spiders there are at least
$$(sk-1)q-(sk-s)q=(s-1)q\ge q$$
spiders $T_p$ such that $[V(T_p)\setminus\{a_2,\ldots,a_s\}]\cap X=\emptyset$. 
Hence, we can continue the process until we find $T_1,\ldots,T_{q}$ such that for each $i\in[q]$ $T_i$ is a spider with leaf vector $(u_i,\va)$ and length vector $(1,k,\dots, k)$ and $V(T_i)\cap V(\cup_{i=1}^q P_i)=\{u_i \}$ and that $V(T_i)\setminus \{a_2,\ldots,a_{s}\}$ are disjoint over all $i\in[q]$. Now $T_1\cup P_1,\ldots,T_q\cup P_q$ are $q$ internally disjoint spiders with leaf vector $(x,\va)$ and length vector $(j,k,\ldots,k)$. The proof of Claim 4 is completed. \qed

\medskip

By Claim 4,  some member of $\cF_{\va,z}$ contains a $(j,k,\dots,k)$-strong sub-spider for some
$2\leq j\leq k$, which contradicts our assumption that no member of $\cF$ contains
any $(\ell_1,\dots,\ell_s)$-strong sub-spiders for any $(\ell_1,\dots, \ell_s)\neq (1,k,\dots, k)$.
This contradiction completes our proof of the lemma.
\end{proof}


\subsection{Proof of Theorem \ref{main}} \label{main-proof}

The main idea of the proof of Theorem \ref{main} is roughly as follows. In an almost regular graph with minimum degree $\delta \geq \Omega(n)$ there are $\Omega(n\delta^{k s})\geq \Omega(n^s)$ balanced $s$-legged spiders of height $k$, that is, copies of $K_{1,s}^{k}$. Using the lemmas in the previous subsection as well as some new ones specific to the $k=3, 4$ cases, we argue that most of these spiders do not contain critical paths of length at most $k$ or any strong sub-spiders. Using the pigeonhole principle, we can find an $s$-tuple that is
the leaf vector of a large number of  $K_{1,s}^{k}$ that do not contain strong sub-spiders
or critical paths of length at most $k$.
This allows us to find at least $t$ copies that are internally disjoint, whose union then forces
a copy of $K_{s,t}^k$.

\begin{lemma}\label{lemma: find t internally disjoint spiders}
	Let $k,s,t,L$ be positive integers. Let $\ell$ be an integer satisfying $s\le\ell\le sk$. Let $\cF_1$ be a family of spiders in a graph $G$ that contain no critical path of length at most $k$ and have the same leaf vector $(v_1,\dots, v_s)$ and length vector $(\ell_1,\dots, \ell_s)$.
	If $|\cF_1|\ge [(sk)^{sk}\cdot f(k,L)^2]^\ell $ then there exists a member of $\cF_1$ that contains a strong sub-spider.
\end{lemma}
\begin{proof}	
	We prove it by induction on $\ell$. The case of $\ell=s$ is trivial. Assume $\ell>s$ and assume that claim holds for smaller $\ell$ values.
	Now pick a maximal family $\cM$ of internally disjoint spiders in $\cF_1$. If $|\cM|\geq  (sk)^{sk-\ell} \cdot f(k,L)$, then any spider in $\cM$ is strong (by Definition \ref{strong-definition}) and we are done. So we may assume $|\cM|< (sk)^{sk-\ell} \cdot f(k,L)$. Let $U$ be the set of internal vertices of spiders in $\cM$. Then $|U|\le sk\cdot |\cM|<(sk)^{sk-\ell+1}\cdot f(k,L)$. By maximality of $\cM$, any spider in $\cF_1$ contains a vertex in $U$. So by averaging, there exists $u\in U$ such that the size of the family $\cF_2$ which consists of all spiders in $\cF_1$ that contain $u$ is at least 
	\[ |\cF_2|\ge\frac{|\cF_1|}{|U|}\ge\frac{|\cF_1|}{(sk)^{sk-\ell+1} \cdot f(k,L)}.\]
	By averaging again, there is a sub-family $\cF_3\subseteq \cF_2$ of size 
	\[|\cF_3|\ge\frac{|\cF_2|}{\ell-s+1}\ge\frac{|\cF_1|}{(sk)^{sk-\ell+2}\cdot f(k,L)}.\]
	such that $u$ plays the same role in members of $\cF_3$. 
	Since any member of $\cF_2$ contains no critical path of length at most $k$ and hence no heavy paths of length at most $k$, there are no more than $\prod_{j=1}^s f(\ell_i,L)\le[f(k,L)]^s$ members of $\cF_2$ that contain $u$ as their center.  
	It is easy to check that by our assumption on $|\cF_1|$ that $|\cF_3|> [f(k,L)]^s$.
	So $u$ cannot be the center of the spiders in $\cF_3$.  Without loss of generality, we assume that $u$ is in the first leg of the spiders in $\cF_3$ and further assume that in every $F\in\cF_3$ the distance between of $u$ and the center of $F$ is $\ell_1'<\ell_1$. Now each member of $\cF_3$ contains a sub-spider with leaf vector $(u,v_2,\ldots,v_s)$ and length vector $(\ell_1',\ell_2,\ldots,\ell_s)$. Let $\cJ$ be the family of sub-spiders with leaf vector $(u,v_2,\ldots,v_s)$ and length vector $(\ell_1',\ell_2,\ldots,\ell_s)$ contained in some member of $\cF_3$. Since the members of $\cF_3$ contain no critical path of length at most $k$, for any $J\in\cJ$ there are no more than $ f(\ell_1-\ell_1',L)\le f(k,L)$ members of $\cF_3$ containing $J$. It follows that
	\[|\cJ|\ge\frac{|\cF_3|}{f(k,L)}\ge\frac{|\cF_1|}{(sk)^{sk}\cdot f(k,L)^2}\geq [(sk)^{sk}\cdot f(k,L)^2]^{\ell-1}.\]
Since $\ell'_1+\ell_2+\dots+\ell_s\leq \ell-1$, and $|\cJ|\geq [(sk)^{sk}\cdot f(k,L)^2]^{\ell-1}$,
 by the induction hypothesis, there exists a member $T$ of $\cJ$ that contains a strong sub-spider. 
Now any member of $\cF_1$ that contain $T$ also contains a strong sub-spider. 
This completes the proof.
\end{proof}

We also need the following lemma that holds only for $k=3,4$. For its proof, let us first recall the definitions of heavy paths and critical paths, given in Section \ref{critical-paths}.

\begin{lemma}\label{lemma: strong --> l_i+l_j>k}
	Suppose that $F$ is an $(\ell_1,\ldots,\ell_s)$-strong spider, where $1\le\ell_1\le\cdots\le\ell_s\le k$. If $F$ contains no critical path of length at most $k$, then $\ell_1+\ell_2\ge k+1$. Moreover, if $k\in\{3,4\}$, then either $\ell_1\ge\frac{k}2$ or $(\ell_1,\ldots,\ell_s)=(1,k,\ldots,k)$.
\end{lemma}
\begin{proof} Let $\ell=\ell_1+\dots +\ell_s$. 
By Definition \ref{heavy-light-path-definition}, every $p$-heavy path contains a $q$-critical path for some $q\leq p$. Since $F$ contains no critical path of length at most $k$, it contains no heavy paths of length at most $k$. Suppose to a contrary that $\ell_1+\ell_2\le k$. Let $(v_1,\ldots,v_s)$ be the leaf vector of $F$. Since $F$ is strong, by Definition \ref{strong-definition}, there are at least 
	$(sk)^{sk-\ell}\cdot f(k,L)>f(k,L)$ internally disjoint spiders with leaf vector $(v_1,\dots, v_s)$
	and length vector $(\ell_1,\dots, \ell_s)$. In particular, in their union, there exist
	at least $f(k,L)\ge f(\ell_1+\ell_2,L)$ internally disjoint paths of length $\ell_1+\ell_2$ joining $v_1$ and $v_2$. This means that the path $P$ in $F$ that joins $v_1$ and $v_2$ is $(\ell_1+\ell_2)$-heavy, contradicting our earlier discussion.
	
	Now assume that $k\in\{3,4\}$ and $\ell_1<\frac{k}{2}$. Then we have that $\ell_1=1$. Since $\ell_1+\ell_j\ge\ell_1+\ell_2\ge k+1$ and $\ell_j\le k$ for every $2\le j\le s$, it follows that $\ell_j=k$. Thus $(\ell_1,\ldots,\ell_s)=(1,k,\ldots,k)$. 
\end{proof}

For $k\in\{3,4\}$,  we can now combine Corollary \ref{general-case-cor} and Lemma \ref{1kk-spiders} to obtain
the following.

\begin{corollary}\label{lemma: k=3 or 4, light spiders}
	Let $k\in\{3,4\}$, $K\ge1$ and integers $s,t\ge2$ be fixed. Then provided that $L$ is sufficiently large compared to $s,t,k$ and $K$, for any $\zeta>0$ there exist $C,n_0>0$ such that the following holds. Suppose that $G$ is an $K_{s,t}^k$-free $K$-almost-regular graph $n\ge n_0$ vertices with minimum degree $\delta\ge Cn^{\frac1k-\frac1{sk}}$. 
	Let $\cF$ denote the family of $s$-legged spiders of $k$ that contain a strong sub-spider but contain no critical path of length at most $k$. Then
	$|\cF|\leq \zeta n\delta^{sk}$.
\end{corollary} 
\bigskip

Now, we are finally ready to prove our main theorem.

\medskip

{\bf Proof of Theorem \ref{main}:}
First we set some constants.
Fix integers $s, t\ge2$ and $k\in\{3,4\}$. 
Let $K$ be obtained by Lemma \ref{almost-regular} with $\epsilon=\frac1k-\frac1{sk}$. 
Let $c=c(k,s)$ such that $\frac1c$ equals the cardinality of the automorphism group of the $s$-legged spider of height $k$. 
Choose $L$ to be a large constant such that Lemma \ref{short-paths} and Corollary \ref{lemma: k=3 or 4, light spiders} are valid. 
We further require that $L$ is large enough such that $\frac{2K^{sk}}{f(1,L)}\le\frac{c}{4k}$. 
Let $C_1=[(sk)^{sk}\cdot f(k,L)^2]^{sk}$, that is, be the constant in Lemma \ref{lemma: find t internally disjoint spiders} with $\ell=sk$. 
Let $C$ be a large constant such that Corollary \ref{lemma: k=3 or 4, light spiders} holds with $\zeta:=\frac{c}8$. We further require that $C$ is large enough such that $\frac{cC^{sk}}{8}\ge C_1$.

By Lemma \ref{almost-regular}, it suffices to show the following statement. 
For sufficiently large $n$, if $G$ is an $n$-vertex $K$-almost-regular graph with minimum degree $\delta\geq C n^{\frac{1}{k}-\frac{1}{sk}}$, then $G$ contains a copy of $K_{s,t}^k$.

We will prove this by contradiction. Suppose to the contrary that $G$ is $K_{s,t}^k$-free. Let $\cF$ be the family of all the $s$-legged spiders of height $k$ in $G$. Then by a greedy process, it is easy to see that 
\begin{equation}\label{equ: size of cF}
|\cF|\geq cn\prod_{i=0}^{sk-1}(\delta-i) \ge \frac{c} 2n \delta^{sk},
\end{equation}
where the last inequality holds because $\delta\geq C n^{\frac{1}{k}-\frac{1}{sk}}$ and $n$ is sufficiently large.
By Lemma \ref{short-paths}, for every $2\le\ell\le k$, the number of critical paths of length $\ell$ is at most $\frac{2}{f(\ell-1,L)}n(K\delta)^{\ell}$. Since the maximum degree of $G$ is at most $K\delta$, the number of members of $\cF$ that contain a critical path of length $\ell$ is at most $\frac{2}{f(\ell-1,L)}n(K\delta)^{\ell}\cdot(K\delta)^{sk-\ell}=\frac{2K^{sk}}{f(\ell-1,L)}n \delta^{sk}\le \frac{2K^{sk}}{f(1,L)}n\delta^{sk}\le \frac{c}{4k}n\delta^{sk}$, where the inequality holds by the choice of $L$. So the number of members of $\cF$ that contain a critical path of length at most $k$ is no more than $(k-1)\cdot\frac{c}{4k}n\delta^{sk}<\frac{c}{4}n\delta^{sk}$. Let $\cF'$ denote the family of members of $\cF$ that contain no critical path of length at most $k$. It follows that 
\begin{equation}\label{equ: size of cF'}
|\cF'|\ge|\cF|-\frac{c}{4}n\delta^{sk}\ge\left(\frac{c}2-\frac{c}{4}\right)n\delta^{sk}\ge\frac{c}{4}n\delta^{sk},
\end{equation}
where in the second inequality we used \eqref{equ: size of cF}.

Let $\cF''$ denote the family of spiders in $\cF'$ that contain no strong sub-spider. By Corollary \ref{lemma: k=3 or 4, light spiders} we have that
$$|\cF'\setminus\cF''|\le\zeta n\delta^{sk}=\frac{c}{8}n\delta^{sk},$$
where the last equality holds by the choice of $\zeta$.
This, together with \eqref{equ: size of cF'}, gives us that
$$|\cF''|=|\cF'|-|\cF'\setminus\cF''|\ge\frac{c}4n\delta^{sk}-\frac{c}{8}n\delta^{sk}=\frac{c}{8}n\delta^{sk}.$$
Since $\delta\ge Cn^{\frac1k-\frac1{sk}}$, it follows that
$$|\cF''|\ge\frac{c}8n\delta^{ks}\ge\frac{c C^{ks}}{8}n^s\ge C_1n^s,$$
where the last inequality holds because of the choice of $C$.
By averaging, there exists a tuple
$(v_1,\dots, v_s)$ of distinct vertices such that the subfamily $\cF_1$ of $\cF''$ that consist of all the members of  $\cF''$ that have leaf vector $(v_1,\ldots,v_s)$ has size at least $$|\cF_1|\ge\frac{|\cF''|}{n^s}\ge \frac{C_1n^s}{n^s}=C_1.$$
Now $\cF_1$ is a family of spiders that  have the same leaf vector and contain no critical path of length at most $k$. Since  $|\cF_1|\ge C_1$, by the definition of $C_1$ given at the beginning of the proof and Lemma \ref{lemma: find t internally disjoint spiders}, there exists a member of $\cF_1\subseteq \cF''$ that contains a strong sub-spider. This contradicts our definition of $\cF''$
and completes the proof.
\qed


\section{Concluding remarks}

\medskip

It is easy to derive from the discussions in Sections \ref{critical-paths}, \ref{general-case}, and
\ref{main-proof} the following weakening of Conjecture \ref{CJL-conjecture}.

\begin{proposition} \label{at-most-k}
Let $s,t,k\geq 2$ be integers. Let $\cK^{\leq k}_{s,t}$ denote the family of graphs that
can be be obtained from $K_{s,t}$ by replacing each edge $uv$ with a path of length at
most $k$ between $u$ and $v$ so that the $st$ replacing paths are internally disjoint.
Then $ex(n,\cK^{\leq k}_{s,t})=O(n^{1+\frac{1}{k}-\frac{1}{sk}})$.
\end{proposition}
This together with the general theorem of Bukh and Conlon \cite{BC} implies the following.
\begin{corollary}
Let $s,t,k\geq 2$ be integers. Then $ex(n,\cK^{\leq k}_{s,t})=\Theta(n^{1+\frac{1}{k}-\frac{1}{sk}})$.
\end{corollary}

\section{Acknowledgements}
The authors thank Jozsef Balogh for his valuable comments on an earlier version of the paper,
in particular, for suggesting the question leading to Proposition \ref{at-most-k}.

\end{document}